\documentclass[12pt, leqno]{article}
\usepackage{amsmath,amsfonts, amssymb, color,hyperref}

\numberwithin{equation}{section}

\newtheorem{theorem}{Theorem}[section]

\newtheorem{lemma}[theorem]{Lemma}

\newtheorem{assume}[theorem]{Assumption}

\newenvironment{proof}{\addvspace{\medskipamount}\par\noindent{\it Proof}.}
{\unskip\nobreak\hfill$\Box$\par\addvspace{\medskipamount}}

\newcommand{\ang}[1]{\langle  #1 \rangle }  

\newcommand{\p}{\mathbb{P}}

\newcommand{\E}[1]{\mathbb{E}\left[#1\right]}

\newcommand{\trieq}{\stackrel{\triangle}{=}}

\newcommand{\cF}{{\mathcal F}}
\newcommand{\cG}{{\mathcal G}}
\newcommand{\cS}{{\mathcal S}}
\newcommand{\R}{\mathbb{R}}

\usepackage{mathrsfs}

\begin{document}

\title{Mixed Deterministic and Random Optimal Control of Linear Stochastic Systems with Quadratic Costs}

\author{Ying Hu\thanks{IRMAR,
Universit\'e Rennes 1, Campus de Beaulieu, 35042 Rennes Cedex, France (ying.hu@univ-rennes1.fr)
and School of
Mathematical Sciences, Fudan University, Shanghai 200433, China.
Partially supported by Lebesgue center of mathematics ``Investissements d'avenir"
program - ANR-11-LABX-0020-01,  by ANR CAESARS - ANR-15-CE05-0024 and by ANR MFG - ANR-16-CE40-0015-01.} \and
Shanjian Tang\thanks{Department of Finance and Control Sciences, School of
Mathematical Sciences, Fudan University, Shanghai 200433, China (e-mail: sjtang@fudan.edu.cn).
Partially supported by National Science Foundation of China (Grant No. 11631004)
and Science and Technology Commission of Shanghai Municipality (Grant No. 14XD1400400).   }}

\maketitle

\begin{abstract}

In this paper, we consider the mixed optimal control of a linear stochastic system with a quadratic cost functional, with two controllers---one can choose only deterministic time functions, called the deterministic controller, while the other can choose adapted random processes, called the random controller. The optimal control is shown to exist under suitable assumptions. The  optimal control is characterized via a system of fully coupled forward-backward stochastic differential equations (FBSDEs) of mean-field type. We solve the FBSDEs via solutions of two (but decoupled) Riccati equations, and give the respective optimal feedback law for both deterministic and random controllers, using solutions of both Riccati equations. The optimal state satisfies a linear stochastic differential equation (SDE) of mean-field type. Both the singular and infinite time-horizonal cases are also addressed.
\end{abstract}



\section{Introduction and formulation of the problem}

Let $T>0$ be given and fixed. Denote by $\cS^n$ the totality of $n\times n$ symmetric matrices, and  by $\cS^n_+$ its subset of  all $n\times n$ nonnegative matrices. We mean by an $n\times n$ matrix $S\ge 0$ that $S\in \cS^n_+$ and by a matrix $S> 0$ that $S$ is positive definite. For a matrix-valued function $R: [0,T]\to \cS^n$, we mean by $R\gg 0$ that $R(t)$ is uniformly positive, i.e. there is a positive real number $\alpha$ such that $R(t)\ge \alpha I$ for any $t\in [0,T].$

In this paper, we consider the following linear control stochastic differential equation (SDE)
\begin{equation}\label{controlsystem}
dX_s=[A_sX_s+B_s^1u_s^1+B_s^2u_s^2]ds+ \sum_{j=1}^d[C_s^jX_s+D_s^{1j} u_s^1+D_s^{2j} u_s^2 ]dW_s^j, \quad s>0;\quad X_0=x_0,
\end{equation}
with the following quadratic cost functional
\begin{equation}\label{cost}
J(u)\trieq\frac{1}{2}\mathbb E\int_0^T\left[ \ang{Q_sX_s, X_s}+\ang{ R_s^1u_s^1, u_s^1}+\ang{ R_s^2u_s^2, u_s^2}\right]ds+\frac{1}{2}\mathbb E [\ang{G X_T, X_T}].
\end{equation}
Here, $(W_t)_{0 \le t \le T}=(W_t^1,\cdots,W_t^d)_{0 \le t \le T}$
is a $d$-dimensional Brownian motion on a probability
space $(\Omega, \cF, \p)$. Denote by $(\cF_t)$ the
augmented filtration generated by $(W_t)$.
$A, B^1,B^2,C^j,D^{1j}$ and $D^{2j}$ are all bounded Borel measurable functions from $[0,T]$
to $\R^{n\times n}, \R^{ n\times l_1}, \R^{ n\times l_2}, \R^{n\times n}, \R^{ n\times l_1}$, and $\R^{ n\times l_2}$,  respectively. $Q, R^1,$ and $ R^2$ are  nonnegative definite,  and they are all essentially bounded measurable functions on $[0,T]$ with values in ${\mathbb S}^n, {\mathbb S}^{l_1}$,
and ${\mathbb S}^{l_2}$,  respectively. In the first four sections, $R^1$ and $ R^2$ are further assumed to be positive definite.   $G\in \mathbb S^n$ is  positive semi-definite. For a
The process $u\in L^2_\cF(0, \, T; \, \R^l)$ is the control, and
$X\in L^2_\cF(\Omega; \, C(0, \, T; \, \R^n))$ is the corresponding state process  with initial value $x_0\in \R^n$.

We will use the following notation:
 $\mathbb S^l$: the set of symmetric $l \times l$ real matrices.
  $L^2_{\cG}(\Omega; \, \R^l)$ the set of random variables $\xi: (\Omega, \cG) \rightarrow (\R^l, {\cal B}(\R^l))$  with $\E{|\xi|^2}<+\infty$.
 $L^\infty_{\cG}(\Omega; \, \R^l)$ is the set of essentially bounded  random variables   $\xi: (\Omega, {\cG}) \to (\R^l, {\cal B}(\R^l))$.
 $L^2_\cF(t, \, T; \, \R^l)$ is the set of $\{\cF_s\}_{s\in [t,T]}$-adapted processes $f=\{f_s: t\leq s\leq T\}$ with $\E{ \int_t^T|f_s|^2\, ds} < \infty$, and denoted by $L^2(t, T; \R^l)$ if the underlying filtration is the trivial one.
 $L^\infty_\cF(t, \, T; \, \R^l)$: the set of essentially  bounded $\{\cF_s\}_{s\in [t,T]}$-adapted processes.
 $L^2_\cF(\Omega; \, C(t, \, T; \, \R^l))$: the set of continuous
         $\{\cF_t\}_{s\in [t,T]}$-adapted processes $f=\{f_s: t\leq s\leq T\}$ with $\E{ \sup_{s\in [t,T]}|f_s|^2\, } < \infty$.
We will often use vectors and matrices in this paper, where all vectors are column vectors. For a matrix $M$,
 $M'$ is its transpose, and $|M|=\sqrt{\sum_{i,j}m_{ij}^2}$ is the Frobenius norm.
Define
\begin{equation}
  B:=(B^1,B^2), \quad D:=(D^1, D^2),\quad  R:=\mbox{\rm diag} (R^1,R^2), \quad u:=((u^1)',(u^2)')';
\end{equation}
and for a matrix $K$ with suitable dimensions and $(t,x,u)\in [0,T]\times \mathbb R^n\times \mathbb R^l$,
\begin{eqnarray}
&& (C_tx+D_t u)dW_t=\sum_{j=1}^d (C_t^jx+D_t^{1j} u^1+D_t^{2j} u^2) dW_t^j; \quad C_t'K:=\sum_{j=1}^d (C_t^j)' K^j; \nonumber\\
&& D_tKD_t:=\sum_{j=1}^d (D_t^j)'KD_t^j, \quad C_t'KD_t:=\sum_{j=1}^d (C_t^j)'KD_t^j, \quad C_t'KC_t:=\sum_{j=1}^d (C_t^j)'KC_t^j.\nonumber   
\end{eqnarray}

If both $u^1$ and $u^2$ are adapted to the natural filtration of the underlying Brownian motion $W$ (i. e., $u^i\in U^i_{\rm ad}= \mathscr{L}^2_{\mathscr{F}}(0,T;\R^{l_i})$ for $i=1,2$),  it is well-known that the optimal control exists and can be synthesized into the following feedback of the state:
\begin{equation}\label{Optfeedback}
 u_t=(R_t+D_t'K_tD_t)^{-1}\left(K_tB_t+C_tK_tD_t\right)' X_t, \quad t\in [0,T].
\end{equation}
Here $K$ solves the following Riccati equation:
\begin{eqnarray}\label{Riccati}
  {d\over ds}K_s&=&A_s'K_s+K_sA_s+C_s'K_sC_s+Q_s\nonumber \\
  &&-(K_sB_s+C'_sK_sD_s)(R_s+D_s'K_sD_s)^{-1}(K_sB_s+C'_sK_sD_s)', \quad s\in [0,T];\\
   K_T&=&G.\nonumber 
\end{eqnarray}
See Wonham~\cite{Wonham}, Haussmann~\cite{Haussmann}, Bismut~\cite{Bismut1,Bismut2}, Peng~\cite{Peng}, and Tang~\cite{Tang1} for more details on the general Riccati equation arising from linear quadratic optimal stochastic control with both state- and control-dependent noises and deterministic coefficients.

In this paper, we consider the following situation: there are two controllers called the deterministic controller and the random controller:  the former can impose a deterministic action  $u^1$ only, i.e., $u^1\in U^1_{\rm ad}=L^2(0,T;\R^{l_1})$;  and the latter can impose a random action $u^2$, more precisely $u^2\in U^2_{\rm ad}=L^2_{\cF}(0,T;\R^{l_2})$. Firstly, we apply the conventional variational technique to characterize the optimal control  via a system of fully coupled forward-backward stochastic differential equations (FBSDEs) of mean-field type. Then we give solution of the FBSDEs with  two (but decoupled) Riccati equations, and derive the respective optimal feedback law for both deterministic and random controllers, using solutions of both Riccati equations. Existence and uniqueness is given to both Riccati equations.  The optimal state is shown to satisfy a linear stochastic differential equation (SDE) of mean-field type. Both the singular and infinite time-horizonal cases are also addressed.

The rest of the paper is organized as follows. In Section 2, we give the necessary and sufficient condition of the mixed optimal Controls via a system of FBSDEs. In Section 3, we synthesize the mixed optimal control into linear closed forms of the optimal state. We derive two (but decoupled) Riccati equations, and study their solvability. We state our main result. In Section 4, we address some particular cases. In Section 5, we discuss singular linear quadratic control cases. Finally in Section 6, we discuss the infinite time-horizonal case.

\section{Necessary and sufficient condition  of mixed optimal Controls}\label{formal-derivation}

Let $u^*$ be a fixed control and $X^*$ be the corresponding state process.
For any $t\in [0, T)$, define  the  processes
$(p(\cdot), (k^j(\cdot))_{j=1,\cdots, d})\in L^2_\cF(0,T;\R^n)\times (L^2_\cF(0,T;\R^n))^d$
as  the unique solution to
\begin{equation} \label{adjoint1general}
\left\{\begin{array}{rcl}
dp(s)&=&\displaystyle -[A_s'p(s)+C_s' k(s)+Q_sX^*_s]\, ds +k'(s)dW_s,\quad s\in[0,T];\\
p(T)&=& G X^*_T.
\end{array}\right.
\end{equation}

The following necessary and sufficient condition can be proved in a straightforward way.

\begin{theorem} Let $u^*$ be the optimal control, and $X^*$ be the corresponding solution.
Then there exists a pair of adjoint processes $(p,k)$ satisfying the BSDE~\eqref{adjoint1general}.
Moreover, the following optimality conditions hold true:
\begin{eqnarray}\label{1 opt}
\mathbb E[(B_s^1)'p(s)+(D_s^{1})'k(s)+R_s^1u^{1*}_s]&=&0,\\
\label{2 opt} (B_s^2)'p(s)+(D_s^{2})'k(s)+R_s^2u^{2*}_s&=&0;
\end{eqnarray}
and they are also sufficient for  $u^*$ to be optimal.
\end{theorem}

\begin{proof}
Using the convex perturbation, we obtain in a straightforward way the equivalent condition of the optimal control $u^*$:
\begin{eqnarray}
\mathbb E\int_0^T\langle (B_s^i)'p(s)+(D_s^{i})'k(s)+R_s^iu^{i*}_s, u^i_s\rangle ds&=&0, \quad \forall u^i\in U_{\rm ad}^i; \quad i=1,2.
\end{eqnarray}
The sufficient condition can be proved in a standard way.
\end{proof}

\section{Synthesis of the mixed optimal control}

\subsection{Ansatz}
Define
\begin{eqnarray}
\overline X:=\mathbb{E}[X], \quad \widetilde X:=X-\overline X; \quad \overline{ u^2}:=\mathbb{E}[u^2], \quad \widetilde {u^2}:=u^2-\overline {u^2}.
\end{eqnarray}
We expect a feedback of the following form
\begin{equation}\label{p eqn}
    p_s=P_1(s)\widetilde{X}_s+P_2(s)\bar{X}_s.
\end{equation}
Applying Ito's formula, we have
\begin{eqnarray}\label{dp}
dp&=&P_1'\widetilde{X}ds+P_1[A_s\widetilde{X}+B^2 \widetilde {u^{2*}}]ds+ P_1[C_sX_s+D_s u_s^* ]dW_s\\
& &+P_2'\bar{X}_sds+P_2[A_s\bar{X}_s+B_s^1u_s^{1*}+B_s^2\overline{u_s^{2*}}]\, ds. \nonumber
\end{eqnarray}

Hence
\begin{equation}\label{k eqn}
k^j(s)=P_1(s)(C_s^jX_s+D_s^{j} u^*_s).
\end{equation}
Define for $i=1,2,$
\begin{eqnarray}
\Lambda_i(S)&:=& R^i+(D^i)'SD^i, \quad S\in \mathbb{S}^n; \\
\widehat\Lambda(S)&:=& \Lambda_1(S)-(D^{1})'SD^{2}\Lambda_2^{-1}(S)(D^2)'SD^1, \quad S\in \mathbb{S}^n;
\end{eqnarray}
and
\begin{eqnarray}
\Theta_i:=(B^2)'P_i+(D^2)'P_1C.
\end{eqnarray}
Plugging equations~\eqref{p eqn} and ~\eqref{k eqn} into the optimality conditions~\eqref{1 opt} and \eqref{2 opt}:
\begin{eqnarray}\label{1eq}
(B_s^1)'P_2(s)\bar{X}_s+(D_s^1)'P_1(s)(C_s\overline X_s+D_s^{1} u^{1*}_s+D_s^{2} \overline{u_s^{2*}})+R_s^1u^{1*}_s&=&0,\\
(B_s^2)'[P_1(s)\widetilde{X}_s+P_2(s)\bar{X}_s]\quad\quad &&\nonumber \\
\label{2eq}+(D_s^{2})'P_1(s)[C_s(\widetilde X_s+\overline X_s)+D_s^{1} u^{1*}_s
+D_s^{2j} u_s^{2*}]+R_s^2u^{2*}_s&=&0;
\end{eqnarray}
From the last equality, we have
\begin{eqnarray}
u^{2*}=-\Lambda_2^{-1}(P_1)[\Theta_1\widetilde X+\Theta_2 \overline X+(D^2)'P_1D^1u^{1*}]
\end{eqnarray}
and consequently
\begin{eqnarray}
\overline {u^{2*}}=-\Lambda_2^{-1}(P_1)[\Theta_2 \overline X+D_2'P_1D_1u^{1*}].
\end{eqnarray}
In view of \eqref{1eq}, we have
\begin{eqnarray}
(B_s^1)'P_2(s)\bar{X}_s+(D_s^{1})'P_1(s)C_s\overline X_s&&\nonumber \\
+(D_s^{1})'P_1(s)D_s^{2} \overline{u_s^{2*}}+\Lambda_1(P_1(s))u^{1*}_s&=&0
\end{eqnarray}
and therefore,
\begin{eqnarray}
\Lambda_1(P_1(s))u^{1*}_s+(B_s^1)'P_2(s)\bar{X}_s+(D_s^{1})'P_1(s)C_s\overline X_s&&\nonumber \\
-(D_s^{1})'P_1(s)D_s^{2}\Lambda_2^{-1}(P_1)[\Theta_2(s) \overline X_s+(D_s^2)'P_1(s)D_s^1u^{1*}_s]&=&0
\end{eqnarray}
or equivalently
\begin{eqnarray}
[\Lambda^1(P_1)-(D^{1})'P_1D^{2}\Lambda_2^{-1}(P_1)(D^2)'P_1D^1] u^{1*}&&\nonumber \\
=-[(B^1)'P_2+(D^{1})'P_1C
-(D^{1})'P_1D^{2}\Lambda_2^{-1}(P_1)\Theta_2] \overline X_s.&&
\end{eqnarray}
We have
\begin{eqnarray}
u^1=M^1\overline X, \quad u^2=M^2\widetilde X+M^3\overline X
\end{eqnarray}
where
\begin{eqnarray}
M^1&:=&-[\Lambda_1(P_1)-(D^{1})'P_1D^{2}\Lambda_2^{-1}(P_1)(D^2)'P_1D^1]^{-1}\nonumber\\
&& \times [(B^1)'P_2+(D^{1})'P_1C
-(D^{1})'P_1D^{2}\Lambda_2^{-1}(P_1)\Theta_2],\\
M^2&:=&-\Lambda_2^{-1}(P_1)\Theta_1,\\
M^3&:=&-\Lambda_2^{-1}(P_1)[\Theta_2+(D^2)'P_1D^1M^1].
\end{eqnarray}
In view of \eqref{dp} and \eqref{adjoint1general}, we have
\begin{eqnarray}
dp&=&P_1'\widetilde{X}ds+P_1[A_s\widetilde{X}+B^2 M^2\widetilde {X}]ds+  k'_sdW_s\\
& &+P_2'\bar{X}_sds+P_2[A_s\bar{X}_s+B_s^1M^1\overline X_s+B_s^2M^3_s\overline X_s]\, ds\nonumber\\
&=&\displaystyle -\biggl\{A_s'(P_1(s)\widetilde{X}_s+P_2(s)\bar{X}_s)+(Q_s+C_s'P_1(s)C_s)(\overline X_s+\widetilde X_s)\\
&&+C_s' P_1(s)[D_s^1M^1_s\overline X_s+D_s^2 (M^2\widetilde {X}+M^3_s\overline X_s)]\biggr\}ds\nonumber \\
&&+  k'_sdW_s.\nonumber
\end{eqnarray}
We expect the following system for $(P_1, P_2)$:
\begin{eqnarray}\label{Riccati 1}
&&P_1'+P_1A+A'P_1+C'P_1C+Q\nonumber\\
&&-(P_1B^2+C' P_1D^2)\Lambda_2^{-1}(P_1)(P_1B^2+C' P_1D^2)'=0,\\
&&P_1(T)=G \nonumber
\end{eqnarray}
and
\begin{eqnarray}
P_2'+P_2A+A'P_2+C'P_1C+Q+C'P_1D^1M^1+C'P_1D^2M^3\nonumber\\
+P_2B^1M^1+P_2B^2M^3=0.
\end{eqnarray}
The last equation can be rewritten into the following one:
\begin{eqnarray}\label{Riccati 2}
P_2'+P_2\widetilde A(P_1)+{\widetilde A}'(P_1)P_2+\widetilde Q(P_1)-P_2\mathcal{N}(P_1)P_2=0,\quad  P_2(T)=G
\end{eqnarray}
where for $S\in \mathbb{S}^n_+$,
\begin{eqnarray}
U(S)&:=& S-SD^2\Lambda_2^{-1}(S)\left(D^2\right)'S; \nonumber\\
\widetilde Q(S)&:=&Q+C'U(S)C-C'U(S)D^1\widehat\Lambda^{-1}(S)(D^1)'U(S)C,\nonumber\\
\widetilde A(S)&:=&A-B^2\Lambda_2^{-1}(S)\left(D^2\right)'SC\nonumber\\
&&-\left[B^1-B^2\Lambda_2^{-1}(S)\left(D^2\right)'SD^1\right]\widehat \Lambda^{-1}(S)\left(D^1\right)'U(S)C,\nonumber\\
\mathcal{N}(S)&:=&B^2\Lambda_2^{-1}(S)\left(B^2\right)'\nonumber\\
&&+\left[B^1-B^2\Lambda_2^{-1}(S)\left(D^2\right)'SD^1\right]\widehat \Lambda^{-1}(S)\left[B^1-B^2\Lambda_2^{-1}(S)\left(D^2\right)'SD^1\right]'.\nonumber
\end{eqnarray}

We have the following representation for $M^1$ and $M^2$:
\begin{eqnarray}
M^1&=&-{\widehat \Lambda}^{-1}(P_1)\left[(B^1)'P_1+(D^1)'U(P_1)C-(D^1)'P_1D^1\Lambda_2^{-1}(P_1)(B^2)'P_2\right],\nonumber\\
M^3&=&-\Lambda_2^{-1}(P_1)\biggl\{(B^2)'P_2+(D^2)'P_1C\nonumber\\
&&-(D^2)'P_1D^1{\widehat \Lambda}^{-1}(P_1)\left[(B^1)'P_1+(D^1)'U(P_1)C-(D^1)'P_1D^1\Lambda_2^{-1}(P_1)(B^2)'P_2\right] \biggr\}.
\end{eqnarray}

\begin{lemma} For  $S\in \mathbb{S}^n_+$, we have $\widetilde Q(S)\ge 0$.
\end{lemma}

\begin{proof} First, we show that $U(S)\ge 0. $ In fact, we have (setting $\widehat {D^2}:=S^{1/2}D^2$)
\begin{eqnarray}
U(S)&=&S-S^{1/2}\widehat{D^2}\left[R^2+\left(\widehat{D^2}\right)'\widehat{D^2}\right]^{-1}\left(\widehat{D^2}\right)'S^{1/2}\\
&\ge&S-S^{1/2}IS^{1/2}=0.
\end{eqnarray}
Here we have used the following well-known matrix inequality:
\begin{eqnarray} \label{matric inequ}
D(R+D'FD)^{-1}D'\le F^{-1}
\end{eqnarray}
for $D\in \mathbb{R}^{n\times m}$, and positive matrices $F\in \mathbb{S}^n$ and $R\in \mathbb{S}^m$.

Using again the inequality~\eqref{matric inequ}, we have (setting $\widehat {D^1}:=[U(S)]^{1/2}D^1$)
\begin{eqnarray}
\widetilde Q(S)&=&Q+C'U(S)C\nonumber\\
&&-C'U(S)D^1\left[R^1+(D^1)'SD^1-(D^1)'SD^2\Lambda_2^{-1}(S)(D^2)'SD^1\right]^{-1}(D^1)'U(S)C\nonumber\\
&=&Q+C'U(S)C-C'U(S)D^1\left[R^1+(D^1)'U(S)D^1\right]^{-1}(D^1)'U(S)C\nonumber\\
&=&Q+C'U(S)C-C'[U(S)]^{1/2}\widehat {D^1}\left[R^1+\left(\widehat {D^1}\right)'\widehat {D^1}\right]^{-1}\left(\widehat {D^1}\right)'[U(S)]^{1/2}C\nonumber\\
&\ge& Q+C'U(S)C-C'[U(S)]^{1/2}I [U(S)]^{1/2}C\ge 0.
\end{eqnarray}
The proof is complete. \end{proof}

\subsection{Existence and uniqueness of optimal control}

\begin{theorem} \label{Main thm1} Assume that $R^1\gg 0$ and $R^2\gg 0$. Riccati equations~\eqref{Riccati 1} and ~\eqref{Riccati 2} have unique nonnegative solutions $P_1$ and $P_2$. The optimal control is unique and  has the following feedback form:
\begin{eqnarray}
u^{1*}=M^1\overline X, \quad u^{2*}=M^2\widetilde X^*+M^3\overline X^*=M^2X^*+(M^3-M^2)\overline X^*.
\end{eqnarray}
Define $\overline X_t^*:=\mathbb{E}[X_t]$ and $\widetilde X_t^*:=X_t^*-\overline X_t^*$
The optimal feedback system is given by
\begin{eqnarray}
X_t&=&x+\int_0^t[(A+B^2M^2)X_s+(B^1M^1-B^2M^2+B^2M^3)\overline X_s]\, ds\nonumber\\
&&+\int_0^t[(C+D^2M^2)X_s+(D^1M^1-D^2M^2+D^2M^3)\overline X_s]\, dW_s, \quad t\ge 0.
\end{eqnarray}
It is a mean-field stochastic differential equation.  The expected optimal state $\overline X_t^*$ is governed by the following ordinary differential equation:
\begin{eqnarray}
\overline X_t&=&x+\int_0^t(A+B^1M^1+B^2M^3)\overline X_s\, ds, \quad t\ge 0;
\end{eqnarray}
and $\widetilde X_t^*$ is governed by the following stochastic differential equation:
\begin{eqnarray}
\widetilde X_t&=&\int_0^t(A+B^2M^2)\widetilde X_s\, ds\nonumber\\
&&+\int_0^t[(C+D^2M^2)\widetilde X_s+(C+D^1M^1+D^2M^3)\overline X_s]\, dW_s, \quad t\ge 0.
\end{eqnarray}
The optimal value is given by
\begin{equation}\label{value}
J(u^*)=\langle P_2(0)X(0), X(0)\rangle.
\end{equation}
\end{theorem}

\begin{proof} Define
\begin{eqnarray}
u^{1*}:=M^1\bar {X^*}, \quad u^{2*}:=M^2\widetilde {X^*}+M^3\bar {X^*}
\end{eqnarray}
 and
\begin{eqnarray}
p^*&=&P_1(s)\widetilde{X^*}+P_2\bar{X^*},\\
k^*&=&P_1[CX^*+D u^*].
\end{eqnarray}
We can check that $(X^*,u^*, p^*,k^*)$ is the solution to FBSDE, satisfying the optimality condition. Hence, $u^*$ is optimal.

If $(X,u, p,k)$ is alternative solution to FBSDE, satisfying the optimality condition, then setting:
$${\delta p}=p-(P_1\widetilde{X}+P_2\bar{X}),\quad {\delta k}=k-P_1(CX+D u).$$
Substituting
$$p={\delta p}+P_1\widetilde{X}+P_2\bar{X},\quad k={\delta k}+P_1(CX+D u)$$
into \eqref{1 opt} and \eqref{2 opt}, we have
\begin{eqnarray}
\label{u1} \mathbb E\left\{(B^1)'({\delta p}+P_1\widetilde{X}+P_2\bar{X})+(D^{1})'[{\delta k}+P_1(CX+D u)]+R^1u^{1}\right\}&=&0,\\
\label{u2} (B^2)'({\delta p}+P_1\widetilde{X}+P_2\bar{X})+(D^{2})'({\delta k}+P_1(CX+D u))+R^2u^{2}_s&=&0.
\end{eqnarray}
From the last equation, we have
\begin{eqnarray}
\overline {u^{2}}=-\Lambda_2^{-1}(P_1)[(B^2)'\overline {\delta p}+(D^2)'\overline {\delta k}+\Theta_2 \overline X+D_2'P_1D_1u^{1*}].
\end{eqnarray}
In view of \eqref{u1} and \eqref{u2}, we have
from the last equation,
\begin{eqnarray}
u^1=L^1\overline {\delta p}+L^2\overline {\delta k}+M^1\overline{X}
\end{eqnarray}
and
\begin{eqnarray}
u^2=L^3\delta p+L^4\delta k+L^5\overline {\delta p}+L^6\overline {\delta k}+M^2\widetilde {X}+M^3\overline X\nonumber
\end{eqnarray}
where
\begin{eqnarray}
L^1&:=&-{\widehat \Lambda}^{-1}(P_1)[(B^1)'-(D^1)'P_1D^2\Lambda_2^{-1}(P_1)(B^2)'],\nonumber\\
L^2&:=&-{\widehat \Lambda}^{-1}(P_1)[(D^1)'-(D^1)'P_1D^2\Lambda_2^{-1}(P_1)(D^2)'],\nonumber\\
L^3&:=&-\Lambda_2^{-1}(P_1)(B^2)',\nonumber\\
L^4&:=& -\Lambda_2^{-1}(P_1)(D^2)',\nonumber\\
L^5&:=&-\Lambda_2^{-1}(P_1)(D^2)'P_1D^1L^1,\nonumber\\
L^6&:=& -\Lambda_2^{-1}(P_1)(D^2)'P_1D^1L^2.\nonumber
\end{eqnarray}
Define the new function $f$ as follows:
\begin{eqnarray}
&&f(s, p, k, P, K)\nonumber\\
&=&[A'_s+P_1(s)B^2_sL^3_s+C'_sP_1(s)D^2_sL^3_s]p+[C'_s+P_1(s)B^2_sL^4+C'P_1(s)D^2_sL^4_s]k\nonumber \\
&&+[C'_sP_1(s)D^1_sL^1_s+P_2(s)B^1_sL^1_s+P_2(s)B^2_sL^3_s-P_1(s)B^2_sL^3_s+P_2(s)B^2_sL^5_s+C'_sP_1(s)D^2_sL^5_s]P\nonumber\\
&&+[C'_sP_1(s)D^1_sL^2_s+P_2(s)B^1_sL^2_s+P_2(s)B^2_sL^4_s-P_1(s)B^2_sL^4_s+P_2(s)B^2_sL^6_s+C'_sP_1(s)D^2_sL^6_s]K. \nonumber
\end{eqnarray}
Then $(\delta p, \delta k)$ satisfies the following linear homogeneous BSDE of mean-field type:
\begin{eqnarray}
d\delta p&=&-f(s, \delta p_s,\delta k_s, \overline{\delta p}_s, \overline{\delta k}_s)\, ds+\delta k dW, \quad \delta p(T)=0.
\end{eqnarray}
In view of Buckdahn, Li and Peng~\cite[Therem 3.1]{BLP}, it admits a unique solution $(\delta p,
\delta k)=(0,0)$.  Therefore, $X=X^*$ and $u=u*$.


The formula~\eqref{value} is derived from  computation of $\langle p_T, X^*_T\rangle$ with the It\^o's formula.
\end{proof}

\section{Particular cases}

\subsection{The classical optimal stochastic LQ case: $B^1=0$ and $D^1=0$.}

In this case, let $P_1$ is the unique nonnegative solution to Riccati equation~\eqref{Riccati 1}. Then, $P_1$ is also the solution of Riccati equation~\eqref{Riccati 2}, and the optimal control reduces to the conventional feedback form.

\subsection{The deterministic control of linear stochastic system with quadratic cost: $B^2=0$ and $D^2=0$.}

In this case, $B=B^1$ and $D=D^1$, and Riccati equation~\eqref{Riccati 1} takes the following form (we write $R=R^1$ for simplifying exposition):
$$P_1'+P_1A+A'P_1+C'P_1C+Q=0,\quad P_1(T)=G,$$
which is a linear Liapunov equation. Riccati equation~\eqref{Riccati 2} takes the following form:
\begin{eqnarray*}& &P_2'+P_2\widetilde A+{\widetilde A}'P_2+\widetilde Q
-P_2B'(R+D'P_1D)^{-1}BP_2=0, \quad P_2(T)=G
\end{eqnarray*}
with
$$
\widetilde  A:=A-B'(R+D'P_1D)^{-1}D'P_1C
$$
and
$$\widetilde Q:=Q+C'P_1C-C'P_1D(R+D'P_1D)^{-1}D'P_1C.
$$
The optimal control takes the following feedback form:
$$u=-(R+D'P_1D)^{-1}(BP_2+D'P_1C)\bar{X}.$$

\section{Some solvable singular cases}

In this section, we study the possibility of $R^1=0$ or $R^2=0$. We have

\begin{theorem} Assume that $R^1\gg 0$ and
\begin{equation}\label{Singular R2}
    R^2\ge 0, \quad (D^2)'D^2\gg 0, \quad G>0.
\end{equation}
Then Riccati equations~\eqref{Riccati 1} and ~\eqref{Riccati 2} have unique nonnegative solutions $P_1\gg 0$ and $P_2$, respectively. The optimal control is unique and  has the following feedback form:
\begin{eqnarray}
u^{1*}=M^1\overline X, \quad u^{2*}=M^2\widetilde X^*+M^3\overline X^*=M^2X^*+(M^3-M^2)\overline X^*.
\end{eqnarray}
The optimal feedback system and the optimal value take identical forms to those of Theorem~\ref{Main thm1}.
\end{theorem}

\begin{proof} In view of the conditions~\eqref{Singular R2}, the existence and uniqueness of solution $P_1\gg 0$ to Riccati equations~\eqref{Riccati 1}
can be found in Kohlmann and Tang~\cite[Theorem 3.13, page 1140]{KT1}, and those of solution $P_2\ge 0$ to Riccati equations~\eqref{Riccati 1}
comes from the fact that $\widehat \Lambda(P_1)\gg 0$ as a consequence of the condition that $R^1\gg 0$.

Other assertions can be proved in an identical manner as Theorem~\ref{Main thm1}.
\end{proof}

\begin{theorem} Assume that $R^2\gg 0$ and
\begin{equation}\label{Singular R1}
    R^1\ge 0, \quad (D^1)'D^1\gg 0, \quad \quad G>0.
\end{equation}
Then Riccati equations~\eqref{Riccati 1} and ~\eqref{Riccati 2} have unique nonnegative solutions $P_1\gg 0$ and $P_2$, respectively. The optimal control is unique and  has the following feedback form:
\begin{eqnarray}
u^{1*}=M^1\overline X, \quad u^{2*}=M^2\widetilde X^*+M^3\overline X^*=M^2X^*+(M^3-M^2)\overline X^*.
\end{eqnarray}
The optimal feedback system and the optimal value take identical forms to those of Theorem~\ref{Main thm1}. .
\end{theorem}

\begin{proof} The existence and uniqueness of solution $P_1$ to Riccati equations~\eqref{Riccati 1}
are well-known. In view of the condition $G>0$, we have $P_1\gg 0$. We now prove those of solution $P_2\ge 0$ to Riccati equations~\eqref{Riccati 1}.

In view of the well-known matrix inverse formula:
\begin{equation}\label{inverse formula}
    \left(A+BD^{-1}C\right)^{-1}=A^{-1}-A^{-1}B\left(D+CA^{-1}B\right)^{-1}CA^{-1}
\end{equation}
for $B\in \R^{n\times m}, C\in \R^{m\times n}$ and invertible matrices $A\in \R^{n\times n}, D\in \R^{m\times m}$ such that  $A+BD^{-1}C$ and $D+CA^{-1}B$ are invertible, we have
the following identity:
\begin{eqnarray}
\widehat{\Lambda}(P_1)&=&R^1+\left(D^1\right)'\left\{P_1-P_1D^2\left[R^2+\left(D^2\right)'P_1D^2\right]^{-1}(D^2)'P_1\right\}D^1\nonumber \\
&=&R^1+\left(D^1\right)'\left[P_1^{-1}+D^2\left(R^2\right)^{-1}\left(D^2\right)'\right]^{-1}D^1.
\end{eqnarray}
Noting the condition $\left(D^1\right)'D^1\gg 0$, we have $\widehat{\Lambda}(P_1)\gg 0$.

Other assertions can be proved in an identical manner as Theorem~\ref{Main thm1}.
\end{proof}

\section{The infinite time-horizontal case}

In this section, we consider the time-invariant situation of all the coefficients $A,B,C,D, Q$ and $R$ in the  linear control stochastic differential equation (SDE)
\begin{equation}\label{controlsystem invariant}
dX_s=[AX_s+B^1u_s^1+B^2u_s^2]ds+ [C_sX_s+D^{1} u_s^1+D^{2} u_s^2 ]dW_s, \quad t>0;\quad X_0=x_0,
\end{equation}
and the quadratic cost functional
\begin{equation}\label{cost invariant}
J(u)\trieq\frac{1}{2}\mathbb E\int_0^\infty\left[ \ang{QX_s, X_s}+\ang{ R^1u_s^1, u_s^1}+\ang{ R^2u_s^2, u_s^2}\right]ds.
\end{equation}

The admissible class of controls for the deterministic controller $u^1$ is $L^2(0,\infty; \mathbb{R}^{l_1})$ and for the random controller $u^2$ is $\mathcal{L}^2_{\mathcal{F}}(0,\infty; \mathbb R^{l_2})$. For simplicity of subsequent exposition, we assume that $Q>0$.

\begin{assume}\label{stable}  There is $K\in \mathbb R^{l_2\times n}$ such that the unique solution $X$ to the following linear matrix stochastic differential equation
\begin{equation}
dX_s=(A+B^2K)X_s\, ds+ (C+D^{2} K)X_sdW_s, \quad t>0;\quad X_0=I,
\end{equation}
lies in $\mathcal{L}^2_{\mathcal{F}}(0,\infty; \mathbb R^{n\times n})$. That is, our linear control system~\eqref{controlsystem invariant} is stabilizable using only control $u^2$.
\end{assume}

We have
\begin{lemma} Assume that $Q>0$ and Assumption~\ref{stable} is satisfied. Then, Algebraic Riccati equations
\begin{eqnarray}\label{Alg Riccati 1}
&&P_1A+A'P_1+C'P_1C+Q\nonumber\\
&&-(P_1B^2+C' P_1D^2)\Lambda_2^{-1}(P_1)(P_1B^2+C' P_1D^2)'=0\nonumber
\end{eqnarray}
and
\begin{eqnarray}\label{Alg Riccati 2}
P_2\widetilde A(P_1)+{\widetilde A}'(P_1)P_2+\widetilde Q(P_1)-P_2\mathcal{N}(P_1)P_2=0
\end{eqnarray}
have positive solutions $P_1$ and $P_2$.
Here for $S\in \mathbb{S}^n_+$,
\begin{eqnarray}
U(S)&:=& S-SD^2\Lambda_2^{-1}(S)\left(D^2\right)'S; \nonumber\\
\widetilde Q(S)&:=&Q+C'U(S)C-C'U(S)D^1\widehat\Lambda^{-1}(S)(D^1)'U(S)C,\nonumber\\
\widetilde A(S)&:=&A-B^2\Lambda_2^{-1}(S)\left(D^2\right)'SC\nonumber\\
&&-\left[B^1-B^2\Lambda_2^{-1}(S)\left(D^2\right)'SD^1\right]\widehat \Lambda^{-1}(S)\left(D^1\right)'U(S)C,\nonumber\\
\mathcal{N}(S)&:=&B^2\Lambda_2^{-1}(S)\left(B^2\right)'\nonumber\\
&&+\left[B^1-B^2\Lambda_2^{-1}(S)\left(D^2\right)'SD^1\right]\widehat \Lambda^{-1}(S)\left[B^1-B^2\Lambda_2^{-1}(S)\left(D^2\right)'SD^1\right]'.\nonumber
\end{eqnarray}
\end{lemma}

\begin{proof} Existence and uniqueness of positive solution $P_1$ to Algebraic Riccati equation~\eqref{Alg Riccati 1} is well-known, and is referred to Wu and Zhou~\cite[Theorem 7.1, page 573]{WuZhou}.
Now we prove the existence of positive solution to Algebraic Riccati equation~\eqref{Alg Riccati 2}. We use approximation method by considering finite time-horizontal Riccati equations.

For any $T>0$, let $P_1^T$ and $P_2^T$ be unique solutions to Riccati equations~\eqref{Riccati 1} and~\eqref{Riccati 2}, with $G=0$. It is well-known that $P_1^T$ converges to the constant matrix $P_1$ as $T\to \infty$. We now show the convergence of $P_2^T$. Firstly, $P_2^T(t)$ is nondecreasing in $T$ for any $t\ge 0$ due to the following representation formula: for $(t,x)\in [0,T]\times \mathbb R^n,$
\begin{equation}\label{value t}
\langle P_2^T(t)x, x\rangle= \inf_{\substack{u^1\in L^2(t,T; \mathbb{R}^{l_1})\\
u^2\in \mathcal{L}^2_{\mathcal{F}}(t,T; \mathbb R^{l_2})}}
\frac{1}{2}\mathbb E^{t,x}\int_t^T\left[ \ang{QX_s, X_s}+\ang{ R^1u_s^1, u_s^1}+\ang{ R^2u_s^2, u_s^2}\right]ds,
\end{equation}
whose proof is identical to that of the formula~\eqref{value}. From~Assumption~\ref{stable}, it is straightforward to show that there is $C_t>0$ such that $|P^T_2(t)|\le C_t$. Then $P^T_2(t)$ converges to $P_2(t)$ as $T\to \infty$. Furthermore,  since all the coefficients are time-invariant and $(P_1^T(T), P_2^T(T))=0$ for any $T>0$, we have
\begin{equation}
\left(P_1^{T+s}(t+s), P_2^{T+s}(t+s)\right)=\left(P_1^T(t), P_2^T(t)\right).
\end{equation}
Taking the limit $T\to \infty$ yields that $P_2(t+s)=P_2(t)$. Therefore, $P_2$ is a constant matrix.

Taking the limit $T\to \infty$ in the integral form of Riccati equation~\eqref{Riccati 2}, we show that $P_2$ solves Algebraic Riccati equation~\eqref{Alg Riccati 2}.

Finally, in view of $Q>0$, we have $P_2^1(0)>0$. Hence $P_2\ge P_2^1(0)>0$.
\end{proof}

\begin{theorem}  \label{Main thm2} Let Assumption~\ref{stable} be satisfied. Assume that $Q>0$ and either of the following three sets of conditions holds true:

(i) $R^1> 0$ and $R^2> 0$;

(ii) $R^1> 0, R^2\ge 0, (D^2)'D^2> 0,$ and $G>0$; and

(iii) $R^1\ge 0, (D^1)'D^1> 0, R^2>  0,$ and $G>0$.

Then the optimal control is unique and  has the following feedback form:
\begin{eqnarray}
u^{1*}=M^1\overline X, \quad u^{2*}=M^2\widetilde X^*+M^3\overline X^*=M^2X^*+(M^3-M^2)\overline X^*.
\end{eqnarray}
Define $\overline X_t^*:=\mathbb{E}[X_t]$ and $\widetilde X_t^*:=X_t^*-\overline X_t^*$
The optimal feedback system is given by
\begin{eqnarray}
X_t&=&x+\int_0^t[(A+B^2M^2)X_s+(B^1M^1-B^2M^2+B^2M^3)\overline X_s]\, ds\nonumber\\
&&+\int_0^t[(C+D^2M^2)X_s+(D^1M^1-D^2M^2+D^2M^3)\overline X_s]\, dW_s, \quad t\ge 0.
\end{eqnarray}
It is a mean-field stochastic differential equation.  The expected optimal state $\overline X_t^*$ is governed by the following ordinary differential equation:
\begin{eqnarray}
\overline X_t&=&x+\int_0^t(A+B^1M^1+B^2M^3)\overline X_s\, ds, \quad t\ge 0;
\end{eqnarray}
and $\widetilde X_t^*$ is governed by the following stochastic differential equation:
\begin{eqnarray}
\widetilde X_t&=&\int_0^t(A+B^2M^2)\widetilde X_s\, ds\nonumber\\
&&+\int_0^t[(C+D^2M^2)\widetilde X_s+(C+D^1M^1+D^2M^3)\overline X_s]\, dW_s, \quad t\ge 0.
\end{eqnarray}
The optimal value is given by
\begin{equation}\label{value 2}
J(u^*)=\langle P_2X(0), X(0)\rangle.
\end{equation}
\end{theorem}
\begin{proof} The uniqueness of the optimal control is an immediate consequence of the strict convexity of the cost functional in both control variables $u^1$ and $u^2$. We now show that $u^*$ is optimal.

For any admissible pair $(u^1,u^2)$, from Theorem~\ref{Main thm1}, we have
$$
J^T(u)\ge \langle P_2^T(0)x,x\rangle.
$$
Therefore, letting $T\to \infty$, we have $J(u)\ge \langle P_2(0)x,x\rangle$.

For $0\le s\le T<\infty$, let $(u^{*,T}, X^{*,T})$ be the optimal pair corresponding to the time-horizon $T>0$, and the associated adjoint process is denoted by $p^T$.
Using It\^o's formula to compute the inner product $\langle p^T, X^{*,n}\rangle$, noting that $p^T_s=P_1^T(s)\widetilde X_s^{*,T}+P_2^T(s)\overline X_s^{*,T}$, we have
\begin{equation}\label{dual s}
\mathbb{E}\left[\langle P_1^T(s)\widetilde X_s^{*,T}+P_2^T(s)\overline X_s^{*,T}, X^{*,T}_s\rangle\right]+ J^s(u^{*,T})=\langle P_2^T(0)x, x\rangle.
\end{equation}
From stability of solutions of stochastic differential equations, we have for any $s>0$,
$$
\lim_{T\to \infty}\mathbb{E}\max_{0\le t\le s}|X_t^{*,T}-X_t^*|^2=0, \quad \lim_{T\to \infty} \mathbb{E}\int_0^s|u_t^{*,T}-u_t^*|^2=0.
$$
Passing to the limit $T\to \infty$ in \eqref{dual s}, we have for any $s\ge 0$
\begin{equation}\label{dual formula}
\mathbb{E}\left[\langle P_1\widetilde X_s^{*}+P_2\overline X_s^{*}, X^{*}_s\rangle\right] +J^s(u^*)=\langle P_2x, x\rangle.
\end{equation}
Since
$$
\mathbb{E}\left[\langle P_1\widetilde X_s^{*}+P_2\overline X_s^{*}, X^{*}_s\rangle\right] =\mathbb{E}\left[\langle P_1\widetilde X_s^{*},\widetilde  X^{*}_s\rangle\right]
+\mathbb{E}\left[\langle P_2\overline X_s^{*}, \overline X^{*}_s\rangle\right]\ge 0,
$$
we have $J^s(u^*)\le \langle P_2x, x\rangle$, and thus $X^*$ is stable and $u^*$ is admissible .

Passing to the limit $s\to \infty$ in \eqref{dual formula}, we have
\begin{equation}
J^s(u^*)=\langle P_2x, x\rangle.
\end{equation}

Finally, the last formula implies the uniqueness of the positive solution to Algebraic Riccati equation~\eqref{Alg Riccati 2}.
\end{proof}

 \end{document}